\documentclass[leqno,sumlimits,intlimits,namelimits,draft]{amsart}

\usepackage{amssymb,latexsym}
\usepackage[numeric,initials,bibtex-style,nobysame]{amsrefs}
\usepackage{amsthm}

\usepackage{a4wide}
\usepackage[USenglish]{babel}

\usepackage{eucal}
\usepackage{mathrsfs}

\def\pg{\mathhexbox278}

\newcommand{\Cinf}{\ensuremath{\mathcal{C}^\infty}}

\newcommand{\D}{\ensuremath{\mathcal{D}}}
\newcommand{\G}{\ensuremath{\mathcal{G}}}
\renewcommand{\S}{\mathscr{S}}

\newcommand{\mb}[1]{\ensuremath{\mathbb{#1}}}
\newcommand{\N}{\mb{N}}

\newcommand{\R}{\mb{R}}
\newcommand{\C}{\mb{C}}

\renewcommand{\d}{\ensuremath{\partial}}
\newcommand{\diff}[1]{\frac{d}{d#1}}

\newfont{\bl}{msbm10 scaled \magstep2}

\newtheorem{theorem}{Theorem}[section]
\newtheorem{lemma}[theorem]{Lemma}
\newtheorem{proposition}[theorem]{Proposition}

\newtheorem{corollary}[theorem]{Corollary}

\theoremstyle{definition}
\newtheorem{remark}[theorem]{Remark}
\newtheorem{example}[theorem]{Example}

\newcommand{\beq}{\begin{equation}}
\newcommand{\eeq}{\end{equation}}

\newcommand{\isom}{\cong}
\newcommand{\col}{\colon}

\newcommand{\FT}[1]{\widehat{#1}}
\newcommand{\F}{\ensuremath{{\mathcal F}}}

\newcommand{\dis}[2]{\langle #1 , #2 \rangle}
\newcommand{\notmid}{\mid\kern-0.5em\not\kern0.5em}

\newcommand{\norm}[2]{{\| #1 \|}_{#2}}

\newcommand{\Norm}[1]{\norm{#1}{}}

\newcommand{\al}{\alpha}
\newcommand{\be}{\beta}
\newcommand{\ga}{\gamma}

\newcommand{\de}{\delta}
\newcommand{\eps}{\varepsilon}

\newcommand{\vphi}{\varphi}

\newcommand{\la}{\lambda}

\newcommand{\sig}{\sigma}

\renewcommand{\Re}{\ensuremath{\mathop{\mathrm{Re}}}}
\renewcommand{\Im}{\ensuremath{\mathop{\mathrm{Im}}}}

\newcommand{\ovl}[1]{\overline{#1}}

\newcommand{\ds}{\displaystyle}

\begin{document}

\pagestyle{plain}

\title{Limits of regularizations for generalized function solutions to the Schr\"odinger equation with `square root of delta' initial value}

\author{G\"unther H\"ormann}

\address{Fakult\"at f\"ur Mathematik\\
Universit\"at Wien, Austria}

\email{guenther.hoermann@univie.ac.at}

\thanks{Research supported by the FWF project P25326}



\date{\today}

\begin{abstract}
We briefly review results on generalized solutions to the Cauchy problem for linear Schr\"odinger-type equations with non-smooth principal part and their compatibility with classical and distributional solutions. In the main part, we study convergence properties of regularized solutions to the standard Schr\"odinger equation with initial values corresponding to `square roots' of Dirac measures in various duals of classical subspaces of the space of continuous functions. In particular, the main result establishes as limit the invariant mean on the space of almost periodic functions as the restriction of the Haar measure on the Bohr compactification of $\R^n$.
\end{abstract}

\maketitle


\section{Introduction}

The motivation to study Schr\"odinger-type linear partial differential operators with non-smooth coefficients can be drawn from at least two fields of mathematical physics: Geophysical models of seismic wave propagation near the earth's core and quantum dynamics of particles in singular potentials. In \cite{Hoermann:11} the basic structures of both types of models were combined into an abstract mathematical formulation and unique existence of solutions to the following Cauchy problem was shown in a setting allowing for discontinuous or distributional coefficients, initial data, and right-hand sides: With $T > 0$ arbitrary one obtains a unique generalized function $u$ on $\R^n \times [0,T]$ solving
\begin{align}
  \d_t u - \mathrm{i}\, \sum_{k=1}^n   \d_{x_k} (c_k 
    \d_{x_k} u) 
    - i V u &= f \label{SCPDE}\\
    u \mid_{t=0} &= g, \label{SCIC} 
\end{align}
where $c_k$ 
 ($k=1,\ldots,n$), 
 $V$, and $f$ are generalized functions on $\R^n \times [0,T]$ and $g$ is a generalized function on $\R^n$. Colombeau-generalized solutions to linear and nonlinear Schr\"odinger equations with constant coefficient principal part have been constructed previously in \cite{Bu:96,Sto:06b,Sto:06a}. The particular case of Schr\"odinger operators with $\delta$-potential is also settled in terms of non-standard analysis in \cite{Albeverio:88}, and a classic approach with quadratic forms and a Friedrichs extension is discussed briefly in \cite[Example 2.5.19]{Thirring:02}.

Differential operators of Schr\"odinger-type with non-smoothness  in the principal symbol arise as \emph{paraxial equations} in models of wave propagation based on narrow-angle symbol approximations and have been applied in various fields of optics or acoustic tomography, but also to seismic wave propagation near the core-mantle boundary inside the earth in \cite{dHHO:08}. The leading-order approximation leads to model equations of Schr\"odinger-type, where the material properties are encoded into the regularity structure of the coefficients in the principal part and in \cite{dHHO:08} a corresponding evolutionary system---meaning unique solvability of the corresponding Cauchy problem---has been established in an $L^2$-setting allowing the coefficients to be of H\"older- or Sobolev-type regularity below log-Lipschitz continuity. This result put the (H\"older or) Sobolev regularity of the solution in relation to the initial data regularity under lowest possible regularity assumptions on the coefficient, which is crucial in the so-called inverse media analysis of geophysics. 

In the context of quantum mechanics one is interested in allowing for the zero-order term $V$ in the  Schr\"odinger equation $\d_t u = i \Delta_x u + i V u$ to model a singular potential.  Moreover, in the classical $L^2$ theory one has initial data $u \!\mid_{t = 0} = u_0$ such that $|u_0|^2$ corresponds to an initial probability density and $|u(.,t)|^2$ is then usually interpreted as the evolved probability density at time $t$.  We may now think of this situation in more general terms as 
$|u_0|^2$ representing a given initial probability \emph{measure} $\mu$ on $\R^n$, i.e., $u_0$ as generalized initial data representing a `regularized square root of a given probability measure', and of $\mu^t := |u(.,t)|^2$ as the time evolved regularized Borel probability measure.
 A result in \cite{Hoermann:11}, reviewed below in Section 2, shows how to construct a Colombeau generalized function whose square is associated with a given probability measure in the sense of distributional shadows.  We may mention that questions about squares of distributional objects as measures arose also in general relativity theory (cf.\ \cite[Section 5.3]{GKOS:01} and \cite{KS:99,Steinbauer:97,Steinbauer:98,SV:06}).  A regularization approach for powers of delta as initial values in semilinear heat equations has been employed in \cite{NPR:05}.

In Section 2 we review the regularization approach to generalized functions in the sense of Colombeau, square roots of probability measures in this framework,  the main result on unique existence of generalized solutions to the Schr\"odinger-type Cauchy problem (\ref{SCPDE}-\ref{SCIC}), and the relation of Colombeau generalized solutions with classical and distributional solution concepts. Section 3 then discusses in detail the convergence properties of solutions corresponding to regularizations of initial values modeling square roots of a Dirac measure in the dual spaces of classical subspaces of the space of continuous functions. The main result is Theorem \ref{mainthm} establishing the (unique) invariant mean on almost periodic functions as the limit.
 

\section{Regularizations, generalized function solutions, and coherence properties} 

In this section, we review the main results of \cite{Hoermann:11}. Before going into details, we recall a few basics from the theory of Colombeau generalized functions. 

The fundamental idea of Colombeau-type regularization methods is to model non-smooth objects by approximating nets of smooth functions, convergent or not, but with \emph{moderate} asymptotics and to identify regularizing nets whose differences compared to the moderateness scale are \emph{negligible}. For a modern introduction to Colombeau algebras we refer to \cite{GKOS:01}. Here we will also make use of constructions and notations from \cite{Garetto:05b}, where  generalized functions based on a locally convex topological vector space $E$ are defined:
Let $E$ be a locally convex topological vector space whose topology is given by the family of seminorms $\{p_j\}_{j\in J}$. The elements of  
$$
 \mathcal{M}_E := \{(u_\eps)_\eps\in E^{(0,1]}:\, \forall j\in J\,\, \exists N\in\N\quad p_j(u_\eps)=O(\eps^{-N})\, \text{as}\, \eps\to 0\}
$$
and
$$
 \mathcal{N}_E := \{(u_\eps)_\eps\in E^{(0,1]}:\, \forall j\in J\,\, \forall q\in\N\quad p_j(u_\eps)=O(\eps^{q})\, \text{as}\, \eps\to 0\},
$$ 
are called \emph{$E$-moderate} and \emph{$E$-negligible}, respectively. With operations defined componentwise, e.g., $(u_\eps) + (v_\eps) := (u_\eps + v_\eps)$ etc., $\mathcal{N}_E$ becomes a vector subspace of $\mathcal{M}_E$.  We define the \emph{generalized functions based on $E$} as the factor space $\G_E := \mathcal{M}_E / \mathcal{N}_E$. If $E$ is a differential algebra then $\mathcal{N}_E$ is an ideal in $\mathcal{M}_E$ and $\G_E$ is  a differential algebra as well.

Particular choices of $E$ reproduce the standard Colombeau algebras of generalized functions. For example, $E=\C$ with the absolute value as norm yields the generalized complex numbers $\G_E = \widetilde{\C}$; for $\Omega \subseteq \R^d$ open, $E=\Cinf(\Omega)$ with the topology of compact uniform convergence of all derivatives provides the so-called special Colombeau algebra $\G_E=\G(\Omega)$. Recall that $\Omega \mapsto \G(\Omega)$ is a fine sheaf, thus, in particular, the restriction $u|_B$ of $u\in\G(\Omega)$ to an arbitrary open subset $B \subseteq \Omega$ is well-defined and yields $u|_B \in \G(B)$. Moreover, we may embed $\D'(\Omega)$ into $\G(\Omega)$ by appropriate localization and convolution regularization. 

If $E \subseteq \D'(\Omega)$, then certain generalized functions can be projected into the space of distributions by taking  weak limits: We say that $u \in \G_E$ is \emph{associated} with $w \in \D'(\Omega)$, if $u_\eps \to w$ in $\D'(\Omega)$ as $\eps \to 0$ holds for any (hence every) representative $(u_\eps)$ of $u$. This fact is also denoted by $u \approx w$. 

Consider open strips of the form $\Omega_T = \R^n \times\, ]0,T[ \subseteq \R^{n+1}$ (with $T > 0$ arbitrary) and  the spaces $E = H^\infty({\Omega_T}) = \{ h \in \Cinf(\Omega_T) : \d^\al h \in L^2(\Omega_T) \; \forall \al\in \N^{n+1}\}$ with the family of (semi-)norms 
$$
  \norm{h}{H^k} = \Big( \sum_{|\al| \leq k} 
    \norm{\d^\al h}{L^2}^2\Big)^{1/2}
   \quad (k\in \N), 
$$
as well as   
$E = W^{\infty,\infty}({\Omega_T}) = \{ h \in \Cinf(\Omega_T) : \d^\al h \in L^\infty(\Omega_T) \; \forall \al\in \N^{n+1}\}$  with the family of (semi-)norms 
$$
  \norm{h}{W^{k,\infty}} = \max_{|\al| \leq k} \norm{\d^\al h}{L^\infty} \quad (k\in \N). 
$$
Clearly, $\Omega_T$  satisfies the strong local Lipschitz property \cite[Chapter IV, 4.6, p.\ 66]{Adams:75}, hence every element of $H^\infty(\Omega_T)$ and $W^{\infty,\infty}(\Omega_T)$ belongs to $\Cinf(\ovl{\Omega_T})$ by the Sobolev embedding theorem \cite[Chapter V, Theorem 5.4, Part II, p.\ 98]{Adams:75}.

In the sequel, we will employ the following notation 
$$
  \G_{L^2}(\R^n \times [0,T]) := \G_{H^\infty({\Omega_T})}
     \quad\text{ and }\quad
   \G_{L^\infty}(\R^n \times [0,T]) := 
   \G_{W^{\infty,\infty}({\Omega_T})}.
$$
Thus, we will represent a generalized function $u \in \G_{L^2}(\R^n \times [0,T])$ by a net $(u_\eps)$ with the moderateness property
$$
    \forall k \, \exists m: \quad 
    \norm{u_{\eps}}{H^k} = O(\eps^{-m}) \quad (\eps \to 0).
$$
If $(\widetilde{u_{\eps}})$ is another representative of $u$, then 
$$
    \forall k \, \forall p: \quad 
    \norm{u_{\eps} - \widetilde{u_{\eps}}}{H^k} = O(\eps^{p}) 
    \quad (\eps \to 0).
$$
Similar constructions and notations are used in case of $E = H^\infty(\R^n)$ and $E = W^{\infty,\infty}(\R^n)$. Note that by Young's inequality (\cite[Proposition 8.9.(a)]{Folland:99}) any standard convolution regularization with a scaled mollifier of Schwartz class provides embeddings $L^2 \hookrightarrow \G_{L^2}$ and $L^p \hookrightarrow \G_{L^\infty}$ ($1 \leq p \leq \infty$).


As an example of a detailed regularization model we recall a result from \cite{Hoermann:11}, announced above in the introduction, on Colombeau generalized positive square roots of arbitrary probability measures, which can serve as initial values in the Cauchy problem (\ref{SCPDE}-\ref{SCIC}).

\begin{proposition}\label{Wurzel} Let $\mu$ be a Borel probability measure on $\R^n$. Choose $\rho \in L^1(\R^n) \cap W^{\infty,\infty}(\R^n)$ to be positive with $\int \rho = 1$ and satisfying $\rho(x) \geq |x|^{-m_0}$  when $|x| \geq 1$ with some $m_0 > n$. Set  $\rho_\eps(x) = \frac{1}{\eps^n}\rho(\frac{x}{\eps})$ and $h_\eps := \mu * \rho_\eps$, then the following hold:

\noindent (i)  $h_\eps$ is positive and the net $(\sqrt{h_\eps})$  represents an element $\phi \in \G(\R^n)$ such that $\phi^2 \approx \mu$;

\noindent (ii) there exists  $g \in \G_{L^2}(\R^n)$  such that $g^2 \approx \mu$ and  the class of $(g_\eps|_{\Omega})$ is equal to $\phi|_\Omega$ in $\G(\Omega)$, or by slight abuse of notation $g|_\Omega = \phi|_\Omega$, for every bounded open subset $\Omega \subseteq \R^n$.
\end{proposition}

\begin{remark}\label{Wurzel_rem} 
For specific choices of $\rho$ in $L^1(\R^n) \cap H^\infty(\R^n)$ such that  $\sqrt{\rho} \in H^\infty(\R^n)$ we could obtain that $(\phi_\eps)$ is also $H^\infty$-moderate and directly defines a square root in $\G_{L^2}(\R^n)$ without having to undergo the cut-off procedure in part (ii) of Proposition \ref{Wurzel} (which, on the other hand, cannot be avoided for general $\rho \in H^\infty$). For example, putting $\rho(x) = c (1+|x|^2)^{-(n+1)/2}$ with a suitable normalization constant $c > 0$ provides such a mollifier. However, the  above formulation leaves more flexibility in adapting the regularization to particular applications.
\end{remark}

We come now to the main existence and uniqueness result for generalized solutions to the Cauchy problem (\ref{SCPDE}-\ref{SCIC}). Recall that a regularization of an arbitrary finite-order distribution which meets the log-type conditions on the coefficients $c_k$ and $V$ in the following statement is easily achieved by employing a re-scaled mollification process as described in \cite{O:89}.

\begin{theorem}\label{exunthm} Let $c_k$ 
 ($k=1,\ldots,n$) and $V$ be generalized functions in $\G_{L^{\infty}}(\R^n \times [0,T])$ possessing representing nets of real-valued functions, $f$ in  $\G_{L^2}(\R^n \times [0,T])$, and $g$ be in $\G_{L^2}(\R^n)$. Suppose\\ 
 (a) $c_k$ ($k=1\ldots,n$) and $V$ are of  log-type, that is,  for some (hence every) representative $(c_{k \eps})$ of $c_k$  and $(V_\eps)$ of $V$ we have 
$\norm{\d_t {c_{k \eps}}}{L^\infty} = O(\log({1}/{\eps}))$ and  $\norm{\d_t {V_{\eps}}}{L^\infty} = O(\log({1}/{\eps}))$ as $\eps \to 0$\\
and\\ 
(b) that the positivity conditions $c_{k \eps}(x,t) \geq c_0$ for all $(x,t) \in \R^n \times [0,T]$, $\eps\in\,]0,1]$, $k=1,\ldots,n$ with some constant $c_0 > 0$ hold (hence  with $c_0 / 2$  for any other representative and small $\eps$).\\[1mm]
Then the Cauchy problem (\ref{SCPDE}-\ref{SCIC})
has a unique solution $u \in \G_{L^2}(\R^n \times [0,T])$. 
\end{theorem}

\begin{remark}[Bohmian flow] If $u$ is a generalized solution to a Schr\"odinger equation according to the above theorem, then we may define the associated \emph{generalized Bohmian current vector field} 
$$
     |u|^2 \d_t + \sum_{k=1}^n \Im (\ovl{u} \, \d_{x_k}\! u) \, \d_{x_k}.
$$
In this way, the approach of Bohmian mechanics can be extended to the case of singular initial data, which cause the current vector field to be non-smooth. For example, the flows for Gaussian regularizations of a $\delta$ initial value have been sketched in \cite[Subsection 6.1]{GR:02} and could be put in the context of generalized flows. Note that with Gaussian wave packets, the limiting behavior at any $t \neq 0$ is $|u_\eps(.,t)|^2 \to 1 / (4 \pi |t|)$ as $\eps \to 0$ (compare also with the observation in \cite[Section 3.3, Example 1]{Rauch:91}).
\end{remark}


In case of smooth coefficients a simple integration by parts argument shows that any solution to the Cauchy problem obtained from the variational method as in \cite[Chapter XVIII, \pg 7, Section 1]{DL:V5}) is a solution in the sense of distributions as well. In addition, the following result from \cite{Hoermann:11} shows further coherence with the Colombeau generalized solution.

\begin{corollary}\label{cor} Let $V$ 
and $c_k$ ($k=1,\ldots,n$) belong to $C^\infty(\Omega_T) \cap L^\infty(\Omega_T)$ with bounded time derivatives of first-order,  $g_0 \in H^1(\R^n)$, and $f_0 \in C^1([0,T],L^2(\R^n))$. Let $u$ denote the unique Colombeau generalized solution to the Cauchy problem (\ref{SCPDE}-\ref{SCIC}), where $g$, $f$ denote standard embeddings of $g_0$, $f_0$, respectively. Then $u \approx w$, where $w \in C([0,T],H^1(\R^n))$ is the unique distributional solution obtained from the variational method. 
\end{corollary}


\section{Limit behavior of solutions for initial value regularizations corresponding to  `square roots' of probability measures}

\subsection{General observations}

We consider a kind of positive square root of the probability measure $\mu$ on $\R^n$ represented by $(\sqrt{\mu \ast \rho_\eps})_{\eps\in\,]0,1]}$, where $\rho$ is a mollifier similarly as in Proposition \ref{Wurzel}, but drop the requirement of smoothness and moderateness of the net $(\rho_\eps)$, since we want to focus here on ``generic convergence properties'' of the regularizations instead of investigating more structural aspects of Colombeau-type differential algebras. We simply assume for the mollifier $\rho$ that 
\beq\tag{M}
   \rho \in L^1(\R^n), \rho \geq 0, \sqrt{\rho} \in L^1(\R^n), \int_{\R^n} \rho(x) \, dx = 1
\eeq
(note that also $\sqrt{\rho} \in L^2(\R)$ is implied by this condition) 
and obtain a standard delta regularization by $\rho_\eps (x) := \frac{1}{\eps^n} \rho(\frac{x}{\eps})$, which satisfies
\begin{multline}\tag{R}
    \mu \ast \rho_\eps \to \mu \text{ as $\eps \to 0$  in $\S'(\R^n)$ as well as weakly (in the sense of probability theory,}\\ \text{ or in distribution) in the  space $M(\R^n)$  of finite complex Borel measures on $\R^n$, i.e.,}\\  \lim_{\eps \to 0} \int \!\! f (x) (\mu \ast \rho_\eps) (x) \, dx = \int \!\! f \, d\mu \text{ for every $f \in C_b(\R^n)$ (bounded continuous functions on $\R^n$)}.
\end{multline}

\begin{remark} 
\begin{trivlist}
\item{(i)} Weak convergence in the sense of probability theory  means convergence with respect to the $\sig(M(\R^n), C_b(\R^n))$-topology defined on $M(\R^n)$ via the dual pair $(M(\R^n), C_b(\R^n))$ with $(\mu,f) \mapsto \int_{\R^n} \! f \,d\mu$ (non-degeneracy of this pairing follows from \cite[Kapitel VIII, Satz 4.6]{Elstrodt:11}). 

\item{(ii)} Recall the following results on the classical normed dual spaces (with $C_0(\R^n)$ denoting the continuous functions on $\R^n$ vanishing at infinity): $C_0(\R^n)' \isom M(\R^n)$ by the Riesz representation theorem (cf.\ \cite[Chapter III, 5.7]{Conway:90}),  $C_b(\R^n)' \isom M(\beta \R^n)$ with $\beta \R^n$ denoting the Stone-\v{C}ech compactification of $\R^n$ (cf.\ \cite[Chapter V, Corollary 6.4]{Conway:90}), which also happens to be  the spectrum (or maximal ideal space) of the Abelian $C^*$-algebra $C_b(\R^n)$ and can be constructed as the weak* closure of $\{ \delta_x \mid x \in \R^n\}$ in $C_b(\R^n)'$.

\item{(iii)} If $\mu = \de$ we have $\rho_\eps \to \de$, but it is easily seen that $\sqrt{\rho_\eps} \to 0$ in the sense of distributions by action on a test function $\vphi$ upon substituting $y = x/\eps$ in $\int \!\sqrt{\rho_\eps(x)}\, \vphi(x)dx = \eps^{n/2} \int\! \sqrt{\rho(y)}\, \vphi(\eps y)dy$ and applying the dominated convergence theorem (thereby using that $\sqrt{\rho} \in L^1$).  Similar effects have also been observed  in the generalized function model of ultrarelativistic Reissner-Nordstr{\o}m fields in \cite[Equations (15) and (17)]{Steinbauer:97} and are typical of so-called model delta net regularizations in the form $\rho_\eps(x) = \rho(x/\eps)/\eps$. However, note that from the construction in \cite[Example 10.6]{O:92} one could instead obtain  an example of a moderate net  $(\psi_\eps)$ of smooth functions on $\R$ satisfying $\psi_\eps \to \delta$ and $\psi_\eps^2 \to \delta$ in $\S'(\R)$ as $\eps \to 0$. 
\end{trivlist}
\end{remark}

Let $u_\eps$ denote the unique $L^2$-solution to a typical instance---or model rather, since here $\rho_\eps$ is no longer required to be smooth---of a regularization of the Cauchy problem (\ref{SCPDE}-\ref{SCIC}) with  initial value $\sqrt{\mu \ast \rho_\eps}$, right-hand side $f_\eps = 0$, constant coefficients $c_k = 1$ ($k=1,\ldots,n$), and potential $V_\eps = 0$, that is
$$
    \d_t u_\eps = i \Delta u_\eps, 
      \quad u_\eps|_{t=0} = \sqrt{\mu \ast \rho_\eps}.
$$
The solution is given by application of the strongly continuous unitary group $U_t := \exp(i t \Delta)$ ($t \in \R$) of operators on $L^2(\R^n)$, with self-adjoint generator $\Delta$ on the domain $H^2(\R^n)$, in the form $u_\eps(t,x) = (U_t \sqrt{\mu \ast \rho_\eps})(x)$.  Here and in the sequel, we will repeatedly apply the Fourier transform and thereby follow H\"ormander's convention \cite[Chapter 7]{Hoermander:V1}. Applying the Fourier transform $\F$ on $L^2(\R^n)$, we have 
\beq\label{solFT}
     \F u_\eps(\xi,t) = \exp(- i t |\xi|^2) \, \F(\sqrt{\mu \ast \rho_\eps})(\xi),
\eeq 
or, in terms of a spatial convolution (cf. \cite[Sections 3.3, 3.4, 4.2, and 4.4]{Rauch:91}), 
\beq\label{solConv}
     u_\eps(.,t) = K(t) * \sqrt{\mu \ast \rho_\eps}, \quad \text{where } 
     K(x,t) = \frac{e^{- \frac{|x|^2}{4 i t}}}{(4 \pi i t)^{n/2}}.
\eeq

For $t \in \R$ let $\mu^t_\eps$ denote the positive measure on $\R^n$ given by the Lebesgue measure with density function $|u_\eps(t,.)|^2$. Unitarity of $U_t$ implies
\begin{multline*}
  \mu^t_\eps(\R^n) = \int_{\R^n} |u_\eps(t,x)|^2\, dx =
   \int_{\R^n} (U_t \sqrt{\mu \ast \rho_\eps})(x) \cdot
     \ovl{(U_t \sqrt{\mu \ast \rho_\eps})(x)}\, dx \\ 
     =  \int_{\R^n} |\sqrt{\mu \ast \rho_\eps(x)}|^2\, dx
     = \int_{\R^n} \mu \ast \rho_\eps(x)\, dx  =
     \int_{\R^n} \int_{\R^n} \rho_\eps(x - y) \, d\mu(y) \, dx\\ =
     \int_{\R^n} \int_{\R^n} \rho_\eps(x - y) \, dx \, d\mu(y) =
     \int_{\R^n} 1 \, d\mu(y) = 1, 
\end{multline*}
hence $\{\mu^t_\eps : t \in \R, \eps \in \, ]0,1]\}$ is a family probability measures on $\R^n$, with $\mu^0_\eps$ having density $\mu \ast \rho_\eps$, and $\Norm{\mu^t_\eps} = 1$ ($t \in \R$, $\eps \in \, ]0,1]$) holds in the Banach space of finite complex Borel measures $M(\R^n)$.  

\subsection{Initial probability delta}

Recall from \eqref{solConv} that we obtain in this case $u_\eps(.,t) = K(t) * \sqrt{\rho_\eps}$.
We observe that for any $t \neq 0$, the net $(u_\eps(t,.))_{\eps \in \, ]0,1]}$ of bounded functions on $\R^n$  converges to $0$ uniformly, since $\sqrt{\rho_\eps} \in L^1(\R^n)$ and the $L^1$-$L^\infty$-estimate for the Schr\"odinger propagator (\cite[\pg 4.4, Theorem 1]{Rauch:91}) implies 
\beq\label{LoneLinf}
  \norm{u_\eps(t,.)}{L^\infty} \leq \frac{\norm{\sqrt{\rho_\eps}}{L^1}}{(4 \pi |t|)^{n/2}} 
    = \frac{\norm{\sqrt{\rho}}{L^1}}{(4 \pi |t|)^{n/2}} \,\eps^{n/2} \to 0  
    \quad (\eps \to 0).
\eeq
Therefore, $\mu^t_\eps \to 0$ as $\eps \to 0$ in $\S'(\R^n)$ and also with respect to the vague topology on $M(\R^n)$, i.e., pointwise as linear functionals on $C_c(\R^n)$  (cf.\ \cite[\pg 30]{Bauer:01}). 
Since $\Norm{\mu^t_\eps} =  1$ for every $\eps \in\, ]0,1[$, the family of linear functionals $H := \{ \mu^t_\eps \mid \eps \in\, ]0,1[ \}$ is equicontinuous (\cite[Exercise 32.5, page 342]{Treves:06}). By density of $C_c(\R^n)$ in $C_0(\R^n)$,  the weak* topology, i.e.,  $\sig(M(\R^n),C_0(\R^n))$, coincides with $\sig(M(\R^n), C_c(\R^n))$ on the equicontinuous set $H$ (\cite[Proposition 32.5, page 340]{Treves:06}), which implies that $\lim_{\eps \to 0}\dis{\mu^t_\eps}{\psi} = 0$ holds for every $\psi \in C_0(\R^n)$ (alternatively, this can be shown directly by splitting the integrals into two parts, one part over the complement of a compact set, where $\sup \psi$ is arbitrarily small, the remaining part on the compact set is estimated using \eqref{LoneLinf}). However, $(\mu^t_\eps)_{\eps \in \, ]0,1]}$ can certainly not be weakly convergent\footnote{sometimes called Bernoulli convergent} in the sense of probability theory, i.e., pointwise as functionals on $C_b(\R^n)$, since the weak limit would  have to be equal to the vague limit, which is $0$, but $\dis{\mu^t_\eps}{1} = \mu^t_\eps(\R^n) = 1 \not\to 0$ as $\eps \to 0$ (see also  \cite[Theorem 30.8]{Bauer:01}). 

To summarize, an initial value regularization with $\mu = \de = \mu_\eps^0$ satisfying (M)  implies that for every $t \neq 0$,
\begin{multline}\tag{W}
    \mu^t_\eps \to 0 \text{ as $\eps \to 0$  in $\S'(\R^n)$, vaguely, and even weak* in } M(\R^n) \isom C_0(\R^n)',\\ \text{but $(\mu^t_\eps)$ does not converge weakly (in the sense of probability theory) in $M(\R)$}.
\end{multline}

\subsubsection{Case study in one spatial dimension by means of elementary analysis} 

The following one-dimensional example illustrates  the failure of weak convergence in a drastic way, but at the same time it leads to the intuition that ``test functions'' on $\R$ possessing limits at $x=\pm \infty$ or integral averages might restore the convergence.
\begin{example}\label{ex} Let $f \in C_b(\R)$ be given by $f(x) = e^{i \log(1 + |x|)}$ ($x \in \R$). If we use the Gaussian mollifier $\rho(x) = {\exp(-x^2/2)}/{\sqrt{2 \pi}}$ in the regularization, then, for any $t \neq 0$, the net $(\dis{\mu_\eps^t}{f})_{0 < \eps \leq 1}$ of complex numbers has uncountably many cluster points in $\C$: Applying an appropriately scaled version of \cite[Section 3.3, Example 1]{Rauch:91} to accommodate for the square root initial value in our Cauchy problem, a routine calculation yields the explicit expression
$$
  |u_\eps(x,t)|^2 = c_\eps(t) \rho(c_\eps(t) x), 
   \text{ where } c_\eps(t) = \frac{\eps}{\sqrt{t^2 + \eps^4}} \to 0 \; (\eps \to 0),
$$
hence, by symmetry of $f$ and $\rho$ and a simple change of variables,
\begin{multline*}
  \dis{\mu^t_\eps}{f} = 
  2  c_\eps(t) \int_0^\infty e^{i \log(1 + x)} \rho(c_\eps(t) x) \, dx =
  2  \int_0^\infty e^{i \log(1 + \frac{y}{c_\eps(t)})} \rho(y) \, dy \\ =
  2  e^{- i \log c_\eps(t)} \int_0^\infty e^{i \log(c_\eps(t) + y)} \rho(y) \, dy,
\end{multline*}
where the last integral converges to $ \gamma := \int_0^\infty e^{i \log y} \rho(y) \, dy = \Gamma(\frac{1+i}{2})/(2 \sqrt{2 \pi}) \neq 0$ as $\eps \to 0$ by dominated convergence; let $\al \in [0, 2\pi[$ and choose a positive real null sequence $(\eps_n)_{n \in \N}$ such that $c_{\eps_n}(t) = \exp(-\al - 2 \pi n)$ (which is in accordance with $c_\eps \to 0$) to obtain the following cluster point
$$
    \lim_{n \to \infty} \dis{\mu^t_\eps}{f} = 
    \lim_{n \to \infty} 2  e^{i (\al + 2 \pi n) } \int_0^\infty e^{i \log(c_{\eps_n}(t) + y)} \rho(y) \, dy =  2 \gamma e^{i \al}.
$$
\end{example}

\subsubsection*{Convergence on bounded functions possessing limits at $\pm \infty$}

We suppose that $t \neq 0$ and an initial value regularization with $\mu = \de$ satisfying (M). One might suspect from the construction of cluster points in Example \ref{ex}, that a limit of $\dis{\mu^t_\eps}{f}$ exists as $\eps \to 0$, if the function $f$ possesses limits as $x \to \pm \infty$. 

\begin{proposition}\label{PropL}
 If $f \in L^\infty(\R)$ is such that both $L_\pm (f) := \lim\limits_{x \to \pm \infty} f(x)$ exist, then 
$$
   \lim_{\eps \to 0}\, \dis{\mu^t_\eps}{f} = \frac{L_-(f) + L_+(f)}{2}.
$$
\end{proposition}

\begin{proof} Let $f$ be as in the hypothesis. We write
$$
     \dis{\mu^t_\eps}{f} =  \int_{-\infty}^{-1} |u_\eps(x,t)|^2 f(x) \, dx 
     + \int_{-1}^{1} |u_\eps(x,t)|^2 f(x) \, dx  + \int_{1}^{\infty} |u_\eps(x,t)|^2 f(x) \, dx
     =: a_\eps + b_\eps + c_\eps
$$
and note that \eqref{LoneLinf} implies $b_\eps \to 0$ as $\eps \to 0$. We will show that $\lim_{\eps \to 0} c_\eps = L_+(f)/2$. The arguments to show $\lim_{\eps \to 0} a_\eps = L_-(f)/2$ are completely analogous, thus the proof will be complete.

Applying \eqref{solConv} in the special case $n = 1$ and upon a simple change of variables, we have
$$
    |u_\eps(x,t)|^2 = \frac{\eps}{4 \pi |t|} \left| 
       \int_\R e^{-i \frac{\eps x z}{2t}} e^{i \frac{\eps^2 z^2}{4 t}} \sqrt{\rho(z)} \, dz
    \right|^2,
$$
which, upon another change of variables in the outermost integral, gives
$$
   c_\eps = \frac{1}{4 \pi |t|} \int_\eps^\infty
   \left| \int_\R e^{-i \frac{r z}{2t}} 
        e^{i \frac{\eps^2 z^2}{4 t}} \sqrt{\rho(z)} \, dz
    \right|^2 f(\frac{r}{\eps}) \, d r =: \frac{1}{4 \pi |t|} \int_\eps^\infty
   \left| h_\eps(r)  \right|^2 f(\frac{r}{\eps}) \, d r.
$$
We observe that $f(r/\eps) \to L_+(f)$ pointwise as $\eps \to 0$ and that a change of variables yields
$$ 
    h_\eps(r) = 2 |t| \int_\R e^{-i r y} 
        e^{i t \eps^2 y^2} \sqrt{\rho(2 t y)} \, dz =
         2 |t| \, \F_{y \to r}(  e^{i t \eps^2 y^2} \sqrt{\rho(2 t y)}) (r),
$$
which converges in $L^2(\R)$ to $h(r) := 2 |t| \, \F(\sqrt{\rho(2 t .)})  (r)$. We estimate
\begin{multline*}
  4 \pi |t| \left| c_\eps - \frac{1}{4 \pi |t|} \int_0^\infty |h(r)|^2 \, dr\, L_+(f) \right| \\ \leq 
     \left| \int_0^\eps |h_\eps(r)|^2 f(\frac{r}{\eps}) \, dr \right| +
       \left| \int_0^\infty \left( |h_\eps(r)|^2 f(\frac{r}{\eps}) 
           - |h(r)|^2 L_+(f) \right) \, dr
       \right|   \\ \leq
         \eps \norm{f}{\infty} \norm{h_\eps}{\infty}^2 + 
           \int_0^\infty \left( \left| |h_\eps(r)|^2 - |h(r)|^2 \right| 
               |f(\frac{r}{\eps})| \right) \, dr 
          +  \int_0^\infty \left( |h(r)|^2 \left| f(\frac{r}{\eps}) -  L_+(f) \right| \right) \, dr\\
    \leq 
      \eps \norm{f}{\infty} \norm{\sqrt{\rho}}{1}^2 + 
      \norm{f}{\infty} \left| \norm{h_\eps}{2}^2 - \norm{h}{2}^2 \right| + 
      \int_0^\infty \left( |h(r)|^2 \left| f(\frac{r}{\eps}) -  L_+(f) \right| \right) \, dr
\end{multline*}
and observe that all terms in the final upper bound tend to $0$ as $\eps \to 0$: This is obvious for the first term, is implied by $L^2$ convergence $h_\eps \to h$ in the second term, and follows from dominated convergence in the third term. Therefore,
$$
   \lim_{\eps \to 0} c_\eps = \frac{1}{4 \pi |t|} \int_0^\infty |h(r)|^2 \, dr\, L_+(f)
$$
and it remains to observe that condition (M) and the fact $h(-x) =  \ovl{h(x)}$ (since $\sqrt{\rho}$ is real)  imply 
$$
  \int_0^\infty |h(r)|^2 \, dr = \frac{1}{2} \norm{h}{2}^2 = 
  \frac{2 |t|}{2} \norm{\F (\sqrt{\rho})}{2}^2 = |t| 2 \pi \norm{\sqrt{\rho}}{2}^2 =
  2 \pi |t| \norm{\rho}{1}^2 = 2 \pi |t|.
$$
\end{proof}

The above result allows for an interpretation in terms of a limit measure concentrated at infinity:  Note that $C_\pm(\R) := \{ f \in C_b(\R) \mid \exists L_-(f) \text{ and } \exists L_+(f)\}$ is isometrically isomorphic to $C([-\infty,\infty])$, where $[-\infty,\infty]$ is the two-point compactification of $\R$; we obtain $C_\pm(\R) ' \isom M([-\infty,\infty])$ by the Riesz representation theorem and hence Proposition \ref{PropL} implies the following statement (with the slight abuse of notation considering $\mu^t_\eps$ as elements in the dual of $C([-\infty,\infty])$).

\begin{corollary} The net $(\mu^t_\eps)_{\eps \in \, ]0,1]}$ has the weak* limit
$\ds{\frac{1}{2}(\delta_{-\infty} + \delta_\infty)}$ in  $M([-\infty,\infty])$.
\end{corollary}

\subsubsection*{Convergence on almost periodic functions}

In the sequel, we still assume that $t \neq 0$ and that the initial value regularization with $\mu = \de$ has the property (M). In trying to find a subclass of functions in $f \in C_b(\R)$, which is substantially different from $C_\pm(\R)$, but allows for the existence of a limit of $\dis{\mu^t_\eps}{f}$ (as $\eps \to 0$), periodic functions come to mind, since an averaging effect in the integrals might produce convergence. 

\begin{example}[Convergence on trigonometric polynomials] \label{ConTrigPol}  Recall that $\dis{\mu^t_\eps}{1} = \mu^t_\eps(\R^n) = 1$. If $f \in C_b(\R) \setminus C_\pm(\R)$ is given by $f(x) = e^{i x \xi}$ with $\xi \in \R \setminus \{0\}$,  then we claim that $\lim_{\eps \to 0}\, \dis{\mu^t_\eps}{f} = 0$. Indeed, using
Equation \eqref{solFT} and that $\sqrt{\rho}$ is real-valued,  we obtain (with the notation $Rg(x) = g(-x)$)
\begin{multline*}
  \dis{\mu^t_\eps}{f} = \F(\mu^t_\eps)(-\xi) = 
    \F(u_\eps(.,t) \, \ovl{u_\eps(.,t)})(-\xi) =  
    \frac{1}{2 \pi} \F(u_\eps(.,t)) \ast \F(\ovl{u_\eps(.,t)}) (-\xi) \\ =
             \frac{1}{2 \pi}\,  (e^{- i t |.|^2}  \F(\sqrt{\rho_\eps})) \ast
                         (e^{i t |.|^2}  R \ovl{\F(\sqrt{\rho_\eps}}))(-\xi)\\ =
              \frac{1}{2 \pi} \int_\R e^{- i t y^2 + i t (-\xi - y)^2} 
              \, \F(\sqrt{\rho_\eps})(y) \, \ovl{\F(\sqrt{\rho_\eps})(\xi + y)} \, dy\\
              =  \frac{1}{2 \pi}\,  \int_\R e^{i t \xi^2 + 2 i t \xi y} 
              \, \F(\sqrt{\rho_\eps})(y) \, \F(\sqrt{\rho_\eps})(- y - \xi) \, dz =
            e^{i t \xi^2}\, \F^{-1}\big(\F(\sqrt{\rho_\eps}) \,
               \F(e^{i \xi .} R \sqrt{\rho_\eps})\big)(2 t \xi)\\
           =  e^{i t \xi^2} \sqrt{\rho_\eps} \ast (e^{i \xi .} R \sqrt{\rho_\eps})(2 t \xi) =
       \frac{e^{i t \xi^2}}{\eps}  \sqrt{\rho(\frac{.}{\eps})} 
            \ast (e^{i \xi .} R \sqrt{\rho(\frac{.}{\eps})})(2 t \xi).
\end{multline*}
Therefore, 
$$
   |\dis{\mu^t_\eps}{f}| \leq \frac{1}{\eps} |\sqrt{\rho(\frac{.}{\eps})}| 
            \ast |R \sqrt{\rho(\frac{.}{\eps})}|(2 t \xi) = \sqrt{\rho}
            \ast R\sqrt{\rho}\left(\frac{2 t \xi}{\eps}\right) \to 0 \quad (\eps \to 0),
$$
since $L^2(\R) \ast L^2(\R) \subset C_0(\R)$ (\cite[14.10.7]{Dieudonne:V2E}).

We conclude that $\dis{\mu^t_\eps}{f}$ converges, if $f$ is a trigonometric polynomial, i.e., $f(x) = \sum_{j=0}^m a_j e^{i x \xi_j}$ with $a_j \in \C$ and $\xi_j \in \R$ ($j = 0, \ldots, m$). Suppose $\xi_0 = 0$ and $\xi_k \neq 0$, if $k \neq 0$, then we have
$$
    \lim_{\eps \to 0} \dis{\mu^t_\eps}{f} = a_0 = \lim_{R \to \infty} \frac{1}{2R} \int_{-R}^R f(x)\, dx,
$$
since $1 \leq k \leq m$ yields $\int_{-R}^R e^{i x \xi_k} \, dx / (2 R)= (e^{i R \xi _k} - e^{-i R \xi_k})/(2 i \xi_k R) \to 0$ as $R \to \infty$.

\end{example}

Motivated by the above example, we consider the $\norm{.}{\infty}$-closure of the subspace of trigonometric polynomials in $C_b(\R)$, which is the space $AP(\R)$ of \emph{almost periodic functions} on $\R$ (cf. \cite[Chapter VI, Theorems 5.7 and 5.17]{Katznelson:04}). We collect a few basic properties of $AP(\R)$:

\begin{trivlist}

\item{(i)} The subspace $AP(\R)$  is, in fact, a (closed Abelian) unital $C^*$ subalgebra of $C_b(\R)$. This follows easily from \cite[Chapter VI, Theorem 5.7]{Katznelson:04} and the fact that $C_b(\R)$ is an Abelian unital $C^*$ algebra.

\item{(ii)} If $f \in AP(\R)$, then  the \emph{mean} 
\beq\label{mean}
   \ds{m(f) := \lim_{R \to \infty} \frac{1}{2 R} \int_{-R}^R f(x) \, dx}
\eeq 
exists and may be computed in the form $m(f) = \lim_{\eta \to 0} F_\eta \ast f$, where $F_\eta (x) := \eta F(\eta x)$, with any $F \in L^1(\R)$ such that $F \geq 0$ and $\int_\R F(x) \, dx = 1$  (cf.\ \cite[Chapter VI, Subsections 5.10 and 5.11]{Katznelson:04}). In particular, if $F$ is the characteristic function of the unit interval $[0,1]$, we obtain 
\beq\label{mean0}
  \forall f \in AP(\R) \col \quad 
    m(f) = \lim_{R \to \infty} \frac{1}{R} \int_0^R f(x)\, dx
\eeq
(which is not true for any $f \in C_b(\R)$ such that the mean $m(f)$ according to \eqref{mean} exists).

\item{(iii)} If $f \in C_b(\R)$ is the function considered in Example \ref{ex}, recall $f(x) = e^{i \log(1 + |x|)}$, then clearly $f \not\in C_\pm(\R)$. Moreover, $f$ is not an almost periodic function on $\R$ (as noted in \cite[22.17, Problems 8b) and 12b)]{Dieudonne:V6E}), because the mean of $f$ does not exist: Direct calculation, using the symmetry of $f$ and the change of variables $1+x = e^s$, gives that $(1+i) \int_{-R}^R f(x)\, dx / (2 R) = \exp(i \log(1+R)) + \frac{\exp(i \log(1+R)) - 1}{R}$, where the second term tends to $0$ (as $R \to \infty$), whereas the first term does not converge. Therefore, we have 
$$
   C_\pm(\R) \cup AP(\R) \subsetneq C_b(\R).
$$

\item{(iv)} It is not difficult to see that, as subspaces of $C_b(\R)$, we have
$$
     C_\pm(\R) \cap AP(\R) = \text{span}\, \{ 1\}, 
$$     
because \cite[Chapter VI, Lemma 5.3]{Katznelson:04} states that for a given almost periodic function $f$ and $\eps > 0$ arbitrary, there is a number $\la > 0$ such that the image $f(\R)$ is contained in the $\eps$-neighborhood of $f(I)$ for any interval $I \subseteq$ of length $\la$; if $f \in C_\pm(\R)$ in addition, choosing the interval $I$ far out to the right shows that the function values of $f(x)$ ($x \in \R$) vary at most by $\eps$ from the limit $L_+(f)$.

\end{trivlist}

\begin{remark}
 If $f \in L^\infty(\R)$ is as in Proposition \ref{PropL}, i.e., the limits $L_\pm(f)$ at $\pm \infty$ exist, then the mean $m(f)$ exists and $m(f) = (L_-(f) + L_+(f))/2$ holds. This is easily seen as follows: Let $\eps > 0$; first note that $L_-(f) = L_+(\check{f})$, if $\check{f}(x) := f(-x)$; thus, we  consider without loss of generality only 
$$
      \frac{1}{R} \int_0^R f(x) dx - L_+(f) 
      = \frac{1}{R} \int_0^r (f(x) - L_+(f)) dx + \frac{1}{R} \int_r^R (f(x) - L_+(f)) dx,
$$  
and choose $0 < r < R$ such that $|f(x) - L_+(f)| \leq \eps/2$, if $x \geq r$, and $r (\norm{f}{\infty} + |L_+(f)|) / R \leq \eps/2$.  
This observation connects Proposition \ref{PropAP} below with the limit formula given in Proposition \ref{PropL}, but note that the proof of the latter required no extra condition on the mollifier $\rho$ and the former is not conclusive for functions in $C_\pm(\R) \not\subseteq AP(\R)$.
\end{remark}

\begin{proposition}\label{PropAP} Suppose, in addition to (M), that 
\beq\tag{MM}
   x \mapsto (1 + x) \sqrt{\rho(x)}  \text{ and }
    x \mapsto \diff{x} \sqrt{\rho(x)} \text{ belong to }  L^1(\R) \cap L^2(\R),
\eeq
then  
$$
    \forall f \in AP(\R): \quad \lim_{\eps \to 0}\, \dis{\mu^t_\eps}{f} = m(f).
$$
\end{proposition}

\begin{proof} As in the beginning of the proof of Proposition \ref{PropL} we start by splitting the integral according to 
$\dis{\mu^t_\eps}{f} =  \int_{-\infty}^{-1} |u_\eps|^2 f \, dx 
     + \int_{-1}^{1} |u_\eps|^2 f \, dx  + \int_{1}^{\infty} |u_\eps|^2  \, dx
     =: a_\eps' + d_\eps + a_\eps$
and recall that \eqref{LoneLinf} immediately implies $d_\eps \to 0$ as $\eps \to 0$. We will first investigate $\lim_{\eps \to 0} a_\eps$, the evaluation of $\lim_{\eps \to 0} a_\eps'$ is completely analogous.

Again similarly to the proof of Proposition \ref{PropL}, we may call on the explicit representation  $|u_\eps(x,t)|^2 = \frac{\eps}{4 \pi |t|} \left| \int_\R e^{-i \frac{\eps x z}{2t}} e^{i \frac{\eps^2 z^2}{4 t}} \sqrt{\rho(z)} \, dz  \right|^2$ to write
$$
   4 \pi |t| a_\eps = \eps \int_1^\infty |h_\eps(\eps x)|^2 f(x) \, dx, \quad \text{where }
   h_\eps(y) =   \int_\R e^{-i \frac{y z}{2t}} 
        e^{i \frac{\eps^2 z^2}{4 t}} \sqrt{\rho(z)} \, dz.   
$$
We note  that $(|h_\eps|^2)' = h_\eps' \ovl{h_\eps} + h_\eps \ovl{h_\eps}' \in L^1(\R)$, since by assumption (MM) we have $\sqrt{\rho}$ and $z \mapsto z \sqrt{\rho(z)}$ in $L^2(\R)$.
Integration by parts then gives 
$$
  4 \pi |t| a_\eps = \underbrace{\lim_{x \to \infty} \eps |h_\eps(\eps x)|^2 
       \int_0^x f(r) \, dr }_{=: b_\eps}  -
  \underbrace{\eps |h_\eps(\eps)|^2 \int_0^1 f(r) \, dr}_{=: n_\eps} - 
  \underbrace{\int_1^\infty \diff{x}\big(\eps |h_\eps(\eps x)|^2\big) 
    \int_0^x f(r) \, dr \, dx}_{=: c_\eps},
$$
where clearly $n_\eps \to 0$ ($\eps \to 0$), since $\norm{h_\eps}{\infty} \leq \norm{\sqrt{\rho}}{1}$. 

We claim that $b_\eps = 0$, which follows from
$$
    b_\eps = \lim_{x \to \infty} \eps |h_\eps(\eps x)|^2 \int_0^x f(r) \, dr =
      \lim_{x \to \infty} \eps x |h_\eps(\eps x)|^2 \, \frac{1}{x}\! \int_0^x f(r) \, dr =
      \lim_{y \to \infty} y |h_\eps(y)|^2 \cdot 
      \lim_{x \to \infty} \frac{1}{x}\! \int_0^x f(r) \, dr,
$$
where the rightmost limit equals $m(f)$ due to \eqref{mean0} and the next to last factor is $0$, since $y \mapsto (1 + y) h_\eps(y)$ is a bounded function by our hypothesis $(MM)$ on $\rho$.

It remains to investigate 
\begin{multline*}
  c_\eps = \int_1^\infty \eps^2 \big( |h_\eps|^2\big)'(\eps x) \int_0^x f(r) \, dr \, dx =
     \int_\eps^\infty \eps \big( |h_\eps|^2\big)'(y)  \int_0^{y/\eps} f(r) \, dr \, dy\\ =
     \int_\eps^\infty y \big( |h_\eps|^2\big)'(y) \,  \frac{1}{y/\eps} \int_0^{y/\eps} f(r) \, dr \, dy,
\end{multline*}
where we note that by \eqref{mean0}, the factor $\frac{1}{y/\eps} \int_0^{y/\eps} f(r) \, dr$ in the final integrand converges to $m(f)$ pointwise as $\eps \to 0$ and is bounded uniformly by $\norm{f}{\infty}$. Furthermore, $h_\eps(y)$ clearly converges pointwise to 
$$ 
   h(y) := \int_\R e^{-i \frac{y z}{2t}}  \sqrt{\rho(z)} \, dz = \F(\sqrt{\rho})(\frac{y}{2t}),
$$
but we need to show that even $y \big( |h_\eps|^2\big)'(y) \to y \big( |h|^2\big)'(y)$ in a sufficiently strong mode of convergence to prove the following ``educated guess'', which we formulate as claim
\beq\tag{C}
   \lim_{\eps \to 0} c_\eps =  m(f) \int_0^\infty y \big( |h|^2\big)'(y)  \, dy.
\eeq
We consider
\begin{multline*}
  |c_\eps - m(f) \int_0^\infty y \big( |h|^2\big)'(y)  \, dy| \leq \\
  \underbrace{\left| 
     \int_0^\infty \Big( y \big( |h_\eps|^2\big)'(y) \,  \frac{1}{y/\eps} 
     \int_0^{y/\eps} f(r) \, dr  - y \big( |h|^2\big)'(y)\, m(f) \Big)\, dy
  \right|}_{\be_\eps}  +
  \underbrace{\left| \int_0^\eps y \big( |h_\eps|^2\big)'(y) \,  \frac{1}{y/\eps} \int_0^{y/\eps} f(r) \, dr \, dy \right|}_{\al_\eps}.
\end{multline*}
Using bounds on the integrand in $\al_\eps$ due to (MM), we have 
$$
   \al_\eps \leq \eps \norm{f}{\infty} \sup_{0 \leq y \leq \eps} |y|\, |\big( |h_\eps|^2\big)'(y)| \leq 
   \eps \norm{f}{\infty} \sup_{y \in \R} |y|\, 2\,  |h_\eps'(y)| \,|h_\eps(y)| \leq
   \eps \norm{f}{\infty} 2 \, \norm{z \sqrt{\rho(z)}}{1} \norm{\sqrt{\rho}}{1}, 
$$
hence $\al_\eps \to 0$ as $\eps \to 0$.

We may insert appropriate ``mixed terms'' in the integrand of $\be_\eps$ and apply the triangle inequality to obtain
\begin{multline*}
  \be_\eps \leq 
  \int_0^\infty \Big| y \big( |h_\eps|^2\big)'(y)  -  y \big( |h|^2\big)'(y) \Big|\, \Big| \frac{1}{y/\eps} \int_0^{y/\eps} f(r) \, dr \Big|\, dy \\
  + \int_0^\infty \Big| y \big( |h|^2\big)'(y) \Big| \,  \Big| \frac{1}{y/\eps} 
     \int_0^{y/\eps} f(r) \, dr  -  m(f) \Big|\, dy =: \ga_\eps + s_\eps,
\end{multline*}
where $s_\eps \to 0$ by dominated convergence thanks to \eqref{mean0}, the bound $\Big| \frac{1}{y/\eps} \int_0^{y/\eps} f(r) \, dr  -  m(f) \Big| \leq 2 \norm{f}{\infty}$, and the fact that $y \mapsto y \big( |h|^2\big)'(y) = 2 \Re(h'(y) \cdot y h(y)) \in L^2(\R) \cdot L^2(\R) \subseteq L^1(\R)$ due to (MM); furthermore, we have
$$
   \ga_\eps \leq \norm{f}{\infty} 
   \int_0^\infty |y \big( |h_\eps|^2\big)'(y)  -  
   y \big( |h|^2\big)'(y)| \, dy, 
$$
where
\begin{multline*} 
   \int_0^\infty |y \big( |h_\eps|^2\big)'(y)  -  
   y \big( |h|^2\big)'(y)| \, dy  = 
   \int_0^\infty |h_\eps'(y)\,  \ovl{y h_\eps(y)}  -  
   h'(y)\,  \ovl{y h(y)}| \, dy\\
   \leq \int_0^\infty |h_\eps'(y)\,  \ovl{y h_\eps(y)}  -  
   h_\eps'(y)\,  \ovl{y h(y)}| \, dy +
   \int_0^\infty |h_\eps'(y)\,  \ovl{y h(y)}  -  
   h'(y)\,  \ovl{y h(y)}| \, dy\\
   \leq \int_0^\infty |h_\eps'(y)| \,  |\ovl{y h_\eps(y)}  -  
    \ovl{y h(y)}| \, dy +
   \int_0^\infty |h_\eps'(y)\,    -  
   h'(y)|\,  |\ovl{y h(y)}| \, dy \\\leq 
   \norm{h_\eps'}{2} \norm{y h_\eps(y) - y h(y)}{2} + 
   \norm{h_\eps'(y)\,    - h'(y)}{2} \norm{y h(y)}{2}\\ \leq
  \norm{z \sqrt{\rho(z)}}{2} \norm{y h_\eps(y) - y h(y)}{2} + 
   \norm{h_\eps'(y)\,    - h'(y)}{2} \norm{(\sqrt{\rho})'}{2} \to 0 \quad (\eps \to 0)
\end{multline*}
by (MM), the formulae $h(y) = \F(\sqrt{\rho})(\frac{y}{2t})$ and $h_\eps(y) = \F_{z \to y}(e^{i \frac{\eps^2 z^2}{4 t}} \sqrt{\rho(z)})(\frac{y}{2t})$, and the exchange between multiplication and derivative by the Fourier transform. Thus, $\ga_\eps \to 0$ and therefore claim (C) is proved, i.e., we obtain in summary
$$
   \lim_{\eps \to 0} a_\eps = \frac{-1}{4 \pi |t|} \lim_{\eps \to 0} c_\eps =
         \frac{- m(f)}{4 \pi |t|}   \int_0^\infty y \big( |h|^2\big)'(y)  \, dy,
$$
and analogously, $\lim_{\eps \to 0} a_\eps' = 
         \frac{- m(f)}{4 \pi |t|}   \int\limits_{-\infty}^0 y \big( |h|^2\big)'(y)  \, dy$. 
Thus, we combine  and arrive at  
$$
   \lim_{\eps \to 0} \dis{\mu^t_\eps}{f}  =
         \frac{- m(f)}{4 \pi |t|}   \int_\R y \big( |h|^2\big)'(y)  \, dy.
$$
It remains to determine the value of the integral, where we apply integration by parts and Parseval's identity, to obtain
\begin{multline*}
    \int_\R y \big( |h|^2\big)'(y)  \, dy = \int_\R y (h'(y) \ovl{h(y)} + h(y) \ovl{h'(y)}) \, dy 
    = - \int_\R h(y) (y \ovl{h(y)})' \, dy + \int_\R  h(y) y \ovl{h'(y)}\, dy\\
    = -  \int_\R h(y) \ovl{h(y)} \, dy - \int_\R  h(y) y \ovl{h'(y)}\, dy + \int_\R  h(y) y \ovl{h'(y)}\, dy = - \int_\R |h(y)|^2 \, dy \\
    = - \int_\R |\F(\sqrt{\rho})(\frac{y}{2t})|^2 \, dy = 
    - 2 |t| \, \norm{\F(\sqrt{\rho})}{2}^2  = - 2 |t| 2 \pi \norm{\sqrt{\rho}}{2}^2 = 
    - 4 |t| \pi \norm{\rho}{1}^2 = - 4 |t| \pi,
\end{multline*}
which completes the proof. 
\end{proof}

We may also give a weak* interpretation of the limit formula in Proposition \ref{PropAP} upon recalling a few facts from the theory of locally compact Abelian groups and Bohr compactifications (cf. \cite[Section 4.7]{Folland:95b}). The \emph{Bohr compactification} $b \R$ of $\R$ is obtained as the group of all (including also the discontinuous) characters on $\R$, i.e., group homomorphisms from $\R$ into the one-dimensional torus group $S^1$, and is equipped with the topology of pointwise convergence, which renders $b \R$ an Abelian compact Hausdorff topological group. The real line $\R$ is continuously embedded into $b \R$ as a dense subgroup, but the embedding is not a homeomorphism onto its image. A function in $C_b(\R)$ is almost periodic, if and only if it is the restriction to $\R$ of a (unique) continuous function on $b \R$; thus, we obtain an isometric isomorphism $AP(\R) \isom C(b\R)$, which in turn implies $AP(\R)' \isom C(b\R)' \isom M(b \R)$ (cf. \cite{Hewitt:53}). By abuse of notation, we consider $\mu^t_\eps$ as elements in $M(b \R)$. We claim that

\begin{multline}\tag{HB} \text{the net $(\mu^t_\eps)_{\eps \in \, ]0,1]}$ converges  to the normalized Haar measure on}\\ \text{the Bohr compactification $b\R$ with respect to the weak* topology in $M(b\R)$.}
\end{multline}
To see this, consider the linear functional $l \col C(b\R) \to \C$, defined by
$$
    l(h) := m(h \!\mid_\R) \quad (h \in C(b\R)).
$$
We clearly have that $l = \text{weak*-}\lim_{\eps \to 0} \mu^t_\eps$, $l$ is continuous (since $|l(h)| \leq \norm{h}{\infty}$), $l$ is positive, i.e., $l(h) \geq 0$ for every nonnegative $h \in C(b\R)$, and that $l$ is normalized, i.e., $l(1) = 1$. It remains to show that $l$ is also translation invariant, i.e., $l (h(. - z)) = l(h)$ for every $z \in b\R$, then the uniqueness of the normalized Haar measure $\la$ on the compact Abelian group $b\R$ in combination with the Riesz representation theorem imply
$$
    \forall h \in C(b \R) \col \quad  l(h) = \int_{b\R} h \, d\la.
$$
Since $b\R$ is compact, the map $z \mapsto h(.-z)$ is continuous $b\R \to C(b\R)$ for every $h \in C(b\R)$ (\cite[Proposition 2.6]{Folland:95b}), hence also the composition $G_h(z) := l(h(.-z))$ defines a continuous map $G_h \col b\R \to \C$. Invariance of $l$ with respect to translations $z \in \R$ follows from \cite[5.13, Equation (5.9)]{Katznelson:04} and means that $G_h(z) = G_h(0)$ for every $z$ in the dense subgroup $\R$ of $b\R$. Therefore, continuity of $G_h$ implies $l(h(.-z)) = G_h(z) = G_h(0) = l(h)$ for every $z \in b\R$, that is, translation invariance of $l$ and hence 
$$
     m(h\!\mid_\R) = \int_{b\R} h \, d\la \quad (h \in C(b \R)).
$$

\begin{remark}\label{hinttoabstract} A theorem in harmonic analysis by Blum-Eisenberg (cf.\  \cite[Theorem 1]{BE:74}) states that a sequence of probability measures $(\nu_k)_{k \in \N}$ on the locally compact Abelian group $G$ is weak* convergent to the Haar measure on the Bohr compactification $bG$ of $G$, if and only if for every nontrivial character $\chi$ on $G$ 
the sequence of Fourier transforms $(\FT{\nu_k}(\chi))_{k \in \N}$ converges to $0$. 
We will take up this line of argument in discussing the higher dimensional case in the following subsection. This implies that, in fact, we could deduce already from the result in Example \ref{ConTrigPol} the convergence of $\mu^t_\eps$ to the (normalized) Haar measure on $b\R$. This gives an independent proof of (HB), without additional regularity assumptions on $\rho$, and as a side effect also shows that $(\mu^t_\eps)_{\eps \in \, ]0,1]}$ is ergodic. Moreover, \cite[Theorem 16.3.1]{Dixmier:77} implies that $\mu^t_\eps$ converges to the unique invariant mean on $AP(\R)$. Thus, we obtain a stronger version of Proposition \ref{PropAP} even without additional requirements on $\rho$.
\end{remark}

\subsubsection{Direct application of the Bohr compactification in higher space dimensions} 

We will make use of the observation made in the previous remark to first prove the $n$-dimensional extension of (HB) and then deduce a generalization of Proposition \ref{PropAP}. In fact, all boils down to applying \cite[Theorem 1]{BE:74} (described in Remark \ref{hinttoabstract}) once the required convergence property of the Fourier transformed measures is established. 

\begin{lemma} \label{FTLemma} 

If $\mu = \de_a$ with arbitrary $a \in \R^n$ and we suppose that the basic condition (M) holds for $\rho$,  then 
$\lim\limits_{\eps \to 0} \F(\mu_\eps^t)(\xi) = 0$ for every $ t \neq 0$ and $\xi \neq 0$.
\end{lemma}

\begin{proof} 

Let $\xi \neq 0$ and $t \neq 0$. Similarly as in Example \ref{ConTrigPol}, noting that $\mu \ast \rho_\eps (x) = \rho_\eps(x - a) =: T_a \rho_\eps (x)$ we obtain (again appealing to Equation \eqref{solFT}, to the fact that $\sqrt{T_a \rho_\eps}$ is real-valued, and employing the notation $Rg(x) = g(-x)$)
\begin{multline*}
  \F(\mu^t_\eps)(\xi) = 
    \F(u_\eps(.,t) \, \ovl{u_\eps(.,t)})(\xi) =  
    \frac{1}{(2 \pi)^n} \F(u_\eps(.,t)) \ast \F(\ovl{u_\eps(.,t)}) (\xi) \\ =
             \frac{1}{(2 \pi)^n}\,  \left(e^{- i t |.|^2}  \F(\sqrt{T_a \rho_\eps})\right) \ast
                         \left(e^{i t |.|^2}  R \ovl{\F(\sqrt{T_a \rho_\eps}})\right)(\xi)\\ =
              \frac{1}{(2 \pi)^n} \int_{\R^n} e^{- i t |y|^2 + i t |\xi - y|^2} 
              \, \F(\sqrt{T_a \rho_\eps})(y) \, 
              \ovl{\F(\sqrt{T_a \rho_\eps})(-\xi + y)} \, dy\\
              =  \frac{1}{(2 \pi)^n}\,  \int_{\R^n} e^{i t |\xi|^2 - 2 i t \dis{\xi}{y}} 
              \, \F(\sqrt{T_a \rho_\eps})(y) \, 
              \F(\sqrt{T_a \rho_\eps})(- y - \xi) \, dy\\ =
            e^{i t |\xi|^2}\, \F^{-1}\big(\F(\sqrt{T_a\rho_\eps}) \,
               \F(e^{i \dis{\xi}{.}} R \sqrt{T_a \rho_\eps})\big)(- 2 t \xi)\\
           =  e^{i t |\xi|^2} \sqrt{T_a \rho_\eps} \ast 
           \left( e^{i \dis{\xi}{.}} R \sqrt{T_a \rho_\eps}\right) (-2 t \xi). 
\end{multline*}

Therefore, we have upon an $\eps$-scaling followed by a translation of the variable of integration,
\begin{multline*}
   |\F(\mu^t_\eps)(\xi)| \leq \int_{\R^n} \sqrt{\rho (x + \frac{a}{\eps})} \, 
   \sqrt{\rho  \left(x + \frac{a}{\eps} + \frac{2 t \xi}{\eps}\right)}  \, dx = 
   \int_{\R^n} \sqrt{\rho (x)} \, 
   \sqrt{\rho  \left(x + \frac{2 t \xi}{\eps}\right)}  \, dx \\=
   (\sqrt{\rho} \ast R \sqrt{\rho}) \big(- \frac{2 t \xi}{\eps}\big) \to 0 
   \quad (\eps \to 0)
\end{multline*}
exactly as in Example \ref{ConTrigPol}, since $L^2(\R^n) \ast L^2(\R^n) \subseteq C_0(\R^n)$ (\cite[14.10.7]{Dieudonne:V2E}).

%
%
%
%
%
\end{proof}

We may again call on the \emph{Bohr compactification} $b \R^n$ of $\R^n$ (cf. \cite[Section 4.7]{Folland:95b}), an Abelian compact Hausdorff topological group, described as in the one-dimensional case mentioned above simply as the group of all characters on $\R^n$ equipped with the topology of pointwise convergence. Then $\R^n$ is continuously embedded into $b \R^n$ as a dense subgroup (but not homeomorphic onto its image). Considering $\mu^t_\eps$ as elements in $M(b \R^n)$, we may then apply\footnote{The result is about sequences of probability measures, but holds also for nets with index set $]0,1]$ (directed downward by $\eps \to 0$), since their convergence may equivalently be checked via sequences $(\eps_k)_{k\in \N}$ with $\eps_k \to 0$.} \cite[Theorem 1]{BE:74} to extract from Lemma \ref{FTLemma} a direct proof of the following 

\begin{proposition}\label{mainprop} \label{BohrCor2} If $\mu = \de_a$ ($a \in \R^n$), then the net $(\mu^t_\eps)_{\eps \in \, ]0,1]}$ converges  to the normalized Haar measure on the Bohr compactification $b\R^n$ with respect to the weak* topology in $M(b\R^n)$. 
\end{proposition}

Following from the general definitions and results in \cite[Sections 16.1-3]{Dixmier:77}, the space $AP(\R^n)$ of \emph{almost periodic functions} on $\R^n$ is defined as the uniform closure of the characters on $\R^n$ in $C_b(\R^n)$, i.e., the uniform closure of  the subspace of trigonometric polynomials also in this case. Moreover, a function in $C_b(\R^n)$ is almost periodic, if and only if it is the restriction to $\R^n$ of a unique continuous function on $b \R^n$, which yields an isometric isomorphism $AP(\R^n) \isom C(b\R^n)$ and implies $AP(\R^n)' \isom C(b\R^n)' \isom M(b \R^n)$. Therefore, we easily obtain from Proposition \ref{BohrCor2} and the statement in \cite[Theorem 16.3.1]{Dixmier:77} on the unique invariant mean $m \col AP(\R^n) \to \C$ an immediate proof of the following

\begin{theorem}\label{mainthm} Suppose that $\mu = \de_a$ ($a \in \R^n$) and $\rho$ satisfies (M), 
then  
$$
    \forall f \in AP(\R^n): \quad \lim_{\eps \to 0}\, \dis{\mu^t_\eps}{f} = m(f).
$$
\end{theorem}

Finally we briefly illustrate why the conclusions of Theorem \ref{mainthm} and Proposition \ref{mainprop}  cannot hold for arbitrary initial probability measures $\mu$ on $\R^n$. 

\begin{remark} The statement in \cite[Theorem 1]{BE:74} is that null convergence of the Fourier transforms $\F(\mu^t_\eps)(\xi)$ at every $\xi \neq 0$ is equivalent to weak* convergence of $\mu^t_\eps$ to the Haar measure. Thus, failure of the former for specific initial probability measures $\mu^0_\eps = \mu$ ($\neq \de_a$) allows to deduce that $\mu^t_\eps$ does not converge to the invariant mean in that case.  For example, let $\mu$ be given by a nonnegative density function $h \in C_c(\R^n)$ (times the Lebesgue measure) and suppose that $\rho \in C_c(\R^n)$ (in addition to (M)). Then we claim that the conclusion of Lemma \ref{FTLemma} cannot hold for $\mu^t_\eps$ constructed from solutions of the Schr\"odinger equation according to the regularization of  $\mu = h \, dx$ via $\rho$; more precisely, we claim that the following holds:

\noindent ($*$) \quad for every $t \neq 0$ there is $\xi \in \R^n$, $\xi \neq 0$, such that $\F(\mu_\eps^t)(\xi) \not\to 0$ ($\eps \to 0$).

\noindent By a calculation similar to that in the beginning of the proof of Lemma \ref{FTLemma},
$$
  \F(\mu^t_\eps)(\xi) = e^{i t |\xi|^2} \sqrt{h \ast \rho_\eps} \ast 
           \left( e^{i \dis{\xi}{.}} R \sqrt{h \ast \rho_\eps}\right) (-2 t \xi). 
$$
Due to uniform convergence $h \ast \rho_\eps \to h$ as $\eps \to 0$ and compactness of supports for all factors in the convolutions, we obtain
$$
       \lim_{\eps \to 0} \F(\mu^t_\eps)(\xi) = e^{i t |\xi|^2} \sqrt{h} \ast 
           \left( e^{i \dis{\xi}{.}} R \sqrt{h}\right) (-2 t \xi).
$$
Suppose ($*$) were false, then the above limit relation implies
$$
  \forall \xi \neq 0: \quad 0 = 
  \int  e^{i \dis{\xi}{x}} \sqrt{h(x)} \sqrt{h(x + 2 t \xi)} \, dx.
$$
But $h$ is a probability density and a continuous functions, hence dominated convergence  yields the contradiction
$$
   0 = \lim_{0 \neq \xi \to 0} \int  e^{i \dis{\xi}{x}} \sqrt{h(x)} \sqrt{h(x + 2 t \xi)} \, dx
   = \int  \sqrt{h(x)} \sqrt{h(x)} \, dx = \int  h(x)  \, dx = 1.
$$
\end{remark}


\newcommand{\SortNoop}[1]{}

\end{document}